\documentclass[12pt,reqno,oneside,a4paper]{amsart}
\usepackage[british]{babel}

\usepackage{graphicx}
\usepackage{amscd}
\usepackage{amsmath}
\usepackage{amsfonts}
\usepackage{amssymb}
\usepackage{comment}
\usepackage{color}
\usepackage[usenames,dvipsnames,table,xcdraw]{xcolor}
\usepackage{pdfsync}
\usepackage{bbm}
\usepackage{dsfont}
\usepackage[colorlinks=true]{hyperref}
\usepackage{enumerate}
\usepackage{booktabs}
\usepackage{float}
\usepackage{wrapfig}
\usepackage{caption}
\usepackage{subcaption}
\usepackage{longtable}
\usepackage{listings}
\usepackage[T1]{fontenc}
\usepackage[scaled]{beramono}
\usepackage{tikz}
\usepackage{pgffor}
\usepackage[all]{xy}

\definecolor{trp}{rgb}{1,1,1}

\definecolor{red}{rgb}{1,0,.2}

\newtheorem{theorem}{Theorem}[section]
\theoremstyle{plain}

\newtheorem{definition}[theorem]{Definition}

\newtheorem{question}[theorem]{Question}

\newtheorem{lemma}[theorem]{Lemma}

\newtheorem{prop}[theorem]{Proposition}
\newtheorem{remark}[theorem]{Remark}

\numberwithin{equation}{section}

\newcommand{\R}{\mathbb{R}}
\newcommand{\N}{\mathbb{N}}

\newcommand{\ii}{\mathbf{i}}
\newcommand{\jj}{\mathbf{j}}

\newcommand{\iih}{\boldsymbol{\hat\imath}}
\newcommand{\jjh}{\boldsymbol{\hat\jmath}}

\newcommand{\iiv}{\overline{\imath}}

\newcommand{\iu}{\underline{i}}
\newcommand{\ju}{\underline{j}}

\usepackage{anysize}

\usepackage{caption}

\definecolor{blue}{rgb}{0,0,1}

\definecolor{red}{rgb}{1,0,.7}

\begin{document}
\title[Calculating box dimension with the method of types]{Calculating box dimension with the method of types}

\author{Istv\'an Kolossv\'ary}
\address{Istv\'an Kolossv\'ary, \newline University of St Andrews,  School of Mathematics and Statistics, \newline St Andrews, KY16 9SS, Scotland} \email{itk1@st-andrews.ac.uk}


\begin{abstract}
This paper presents a general procedure based on using the method of types to calculate the box dimension of sets. The approach unifies and simplifies multiple box counting arguments. In particular, we use it to generalize the formula for the box dimension of self-affine carpets of Gatzouras--Lalley and of Bara\'nski type to their higher dimensional sponge analogues. In addition to a closed form, we also obtain a variational formula which resembles the Ledrappier--Young formula for Hausdorff dimension.
\end{abstract}

\maketitle

\thispagestyle{empty}
\section{Introduction}\label{sec:01}

The \emph{box dimension} of a subset $\Lambda$ of $\R^d$ is defined as the limit
\begin{equation*}
\dim_{\mathrm B}\Lambda  = \lim_{\delta\to 0} \frac{\log N_{\delta}(\Lambda)}{-\log \delta},
\end{equation*}
where $N_{\delta}(\Lambda)$ denotes the minimum number of $d$-dimensional boxes of sidelength $\delta$ needed to cover $\Lambda$. More precisely, one takes the $\liminf$ and $\limsup$ to get the lower and upper box dimensions, respectively, but for all sets considered in this paper the limit exists. The main aim of this paper is to provide a unified approach based on the `method of types' to calculate the box dimension. The effectiveness of the argument is demonstrated on various families of self-affine sponges in $\R^d$. Thanks to the flexibility of the method, one can hope to apply it to more complicated constructions in the future and gain additional insight as to when does the Hausdorff and box dimension of a set differ.  

The outline of the general argument goes as follows. Assume that at scale $\delta$ we are given a collection $\mathcal{B}_\delta$ of sets of diameter $\delta$ that is a cover of $\Lambda$ with cardinality $\#\mathcal{B}_\delta=N_{\delta}(\Lambda)$. The first step is to partition $\mathcal{B}_\delta$ into type classes according to some rule. Let $\mathcal{T}_{\delta}$ denote the set of all possible types and $T_\delta^*$ be the class with the most elements. Then    
\begin{equation}\label{eq:00}
\#T_\delta^* \leq N_{\delta}(\Lambda) \leq \#T_\delta^* \cdot \#\mathcal{T}_{\delta}.
\end{equation}
If $\#\mathcal{T}_{\delta}=o(\delta^{-1})$ and $\#T_\delta^*$ has lower and upper bounds such that after taking logarithm, dividing by $-\log \delta$ and letting $\delta\to 0$ we get the same limit for the lower and upper bound, then the growth rate of $\#T_\delta^*$ essentially determines $\dim_{\mathrm B}\Lambda$. We refer to $T_\delta^*$ as the \emph{dominant box counting class} at scale $\delta$ and the type it corresponds to as the \emph{dominant box counting type}. The optimal $\delta$-cover of all sets considered here have a clear symbolic representation which allows us to apply the method of types with proper adaptations.  



The \emph{method of types} is an elementary tool to give good estimates for the number of sequences of a given length with prescribed digit frequencies where the digits come from a finite alphabet. It has roots dating back to works of Boltzmann, Hoeffding, Sanov or Shannon to name a few. It was later systematically developed to study discrete memoryless systems in information theory and has since found applications in for example hypothesis testing, combinatorics, or large deviations, see \cite{Bremaud2017MethodofTypes, Csiszar_98MethodofTypes} for some background. 

Let us recall the basic notions and facts from the method of types that we will use. Let $\mathcal{I}=\{1,\ldots,N\}$ be the finite alphabet and $\Sigma=\mathcal{I}^{\N}$ be the set of all infinite sequences $\ii=(i_1i_2\ldots)$. For any $n\in\N$, we use the notation $\ii|n=i_1\ldots i_n$. 

The \emph{type of $\ii$ at level $n$} is the empirical vector
\begin{equation*}
\tau_n(\ii) = \frac{1}{n}\big( \#\{1\leq\ell\leq n:\, i_{\ell}=1\}, \ldots, \#\{1\leq\ell\leq n:\, i_{\ell}=N\}  \big),
\end{equation*}
that is $\tau_n(\ii)$ just tabulates the relative frequency of each symbol of $\mathcal{I}$ in $\ii|n$. The set of all possible types at level $n$ is
\begin{equation*}
\mathcal{T}_n=\big\{\mathbf{p}: \text{ there exists } \ii\in\Sigma \text{ such that } \mathbf{p}=\tau_n(\ii)\big\}.
\end{equation*}
Let $\mathcal{P}_N$ denote all probability vectors $\mathbf{p}=(p_1,\ldots,p_N)$. Observe that as $n\to \infty$ the set $\mathcal{T}_n$ becomes dense in $\mathcal{P}_N$. A cylinder set is defined as $[i_1\ldots i_n] = \{\jj\in\Sigma:\, \jj|n=(i_1\ldots i_n)\}$. Then $\mathcal{B}_n=\{[i_1\ldots i_n]:\, (i_1,\ldots,i_n)\in\mathcal{I}^n\}$ gives a partition of $\Sigma$. We simply identify the elements of $\mathcal{B}_n$ with finite sequences $(i_1,\ldots,i_n)$. The \emph{type class of} $\mathbf{p}\in \mathcal{T}_n$ is the set
\begin{equation*}
T_n(\mathbf{p}) = \big\{(i_1,\ldots, i_n)\in\mathcal{B}_n:\, \tau_n((i_1,\ldots,i_n))=\mathbf{p}\big\}.
\end{equation*} 
Throughout, we will only use the following two simple facts from the method of types:
\begin{equation}\label{eq:01}
\# \mathcal{T}_n \leq (n+1)^N
\end{equation}
and
\begin{equation}\label{eq:02}
(n+1)^{-N} e^{n H(\mathbf{p})} \leq \#T_n(\mathbf{p}) \leq e^{n H(\mathbf{p})}
\end{equation}
for every $\mathbf{p}\in\mathcal{T}_n$, where $H(\mathbf{p}) = -\sum_i p_i\log p_i$ is the entropy of the probability vector $\mathbf{p}$. For a proof of these elementary facts, we refer to \cite[Lemmas 2.1.2 and 2.1.8]{DemboZeitouniLDP}. Inequality~\eqref{eq:01} implies that it is indeed enough to consider the dominant box counting class, while~\eqref{eq:02} ensures that we get matching lower and upper bounds for $\dim_{\mathrm B}\Lambda$. 

\subsection*{Main contribution}
The idea of picking out classes of words from a code space in some optimal way has been used before, however, the author is unaware of it being formalised in such a general context previously to calculate the box dimension. The main result is to determine the box dimension of Gatzouras--Lalley and of Bara\'nski sponges in arbitrary dimensions. The key technical contribution is to adapt~\eqref{eq:01} and~\eqref{eq:02} to more complicated settings where multi-dimensional types are used for sequences of varying lengths.

\subsection*{Structure of paper}
We begin by demonstrating the skeleton of the argument in the simplest case of self-similar sets satisfying the open set condition which we later build upon. Section~\ref{sec:2} provides a brief overview of related literature on dimensions of self-affine sponges and states our main results, see Theorems~\ref{thm:main} and~\ref{thm:main2}.  The proofs are presented in Sections~\ref{sec:3} and~\ref{sec:4}. In Section~\ref{sec:5}, we discuss possible generalizations of the approach and connections with the Ledrappier--Young formula for the Hausdorff dimension. 

\subsection{Self-similar sets}\label{sec:11}
In general, an \emph{iterated function system} (IFS) on $\R^d$ is a finite family $\mathcal{S}=\{S_1,\ldots,S_N\}$ of contractions $S_i:\, \R^d\to\R^d$. The IFS determines a unique, non-empty compact set $\Lambda$, called the \emph{attractor}, that satisfies the relation
\begin{equation*}
\Lambda = \bigcup_{i=1}^N S_i(\Lambda).
\end{equation*}
In particular, if the maps are similarities, i.e. for every $x,y\in\R^d$
\begin{equation*}
\|S_i(x)-S_i(y)\| = \lambda_i \|x-y\|, \;\text{ where } 0<\lambda_i<1 \text{ is the contraction ratio of } S_i, 
\end{equation*}
then the IFS and its attractor are called \emph{self-similar}. The IFS satisfies the \emph{open set condition} (OSC) if there exists a non-empty open set $V$ such that
\begin{equation}\label{eq:06}
S_i(V) \subseteq V \text{ and } S_i(V)\cap S_j(V) = \emptyset \text{ for } i\neq j.
\end{equation}
It is well-known that a self-similar set has equal Hausdorff and box dimension, moreover, if the OSC is also satisfied then the dimension is given by the Hutchinson formula, i.e. the unique solution $s$, often called the \emph{similarity dimension}, to the equation
\begin{equation}\label{eq:07}
\sum_{i=1}^N \lambda_i^s=1.
\end{equation}
We now sketch the argument for deriving the box dimension using the method of types.

Let $\lambda_{\min} := \min_i \lambda_i,\, \lambda_{\max} := \max_i \lambda_i$ and denote the Lyapunov-exponent with respect to $\mathbf{p}$ by $\chi(\mathbf{p}):=-\sum_ip_i\log \lambda_i$. Throughout, we use the convention that $a\lessapprox b$ if there is an independent constant $C$ such that $a\leq Cb$, similarly $a\gtrapprox b$ if $a\geq Cb$ and $a\approx b$ if $a\lessapprox b$ and $a\gtrapprox b$. The set of finite length words from the alphabet $\mathcal{I}=\{1,\ldots,N\}$ is denoted by $\Sigma^*$ and the length of $\iiv\in\Sigma^*$ is $|\iiv|$.

On the symbolic space $\Sigma$, the \emph{$\delta$-stopping of $\ii\in\Sigma$} is the unique integer $L_{\delta}(\ii)$ such that
\begin{equation}\label{eq:03}
\prod_{\ell=1}^{L_{\delta}(\ii)} \lambda_\ell \leq \delta < \prod_{\ell=1}^{L_{\delta}(\ii)-1}\lambda_\ell, \;\text{ i.e. } L_{\delta}(\ii) \approx \frac{\log \delta}{\frac{1}{L_{\delta}(\ii)} \sum_{\ell=1}^{L_{\delta}(\ii)} \log \lambda_\ell}.
\end{equation}
The \emph{symbolic $\delta$-approximate ball} containing $\ii\in\Sigma$ is
\begin{equation*}
B_\delta(\ii) = \big\{ \jj\in\Sigma:\, \ii|L_{\delta}(\ii) = \jj|L_{\delta}(\ii) \big\},
\end{equation*}
which we identify with the finite sequence $(i_1,i_2,\ldots,i_{L_{\delta}(\ii)})$. The name comes from the fact that the image $\pi(B_\delta(\ii))$ on $\Lambda$ has diameter $\approx \delta$, where $\pi:\Sigma\to\Lambda$ is the natural projection defined by
\begin{equation*}
\pi(\ii) = \lim_{n\to\infty} S_{i_1}\circ S_{i_2}\circ\ldots\circ S_{i_n}(0).
\end{equation*}
The symbolic Moran-cover of $\Sigma$ at scale $\delta$ is $\mathcal{B}_\delta = \big\{\iiv\in\Sigma^*:\, (\forall\, \ii\in[\iiv])\, L_{\delta}(\ii) = |\iiv|\big\}$. It is straightforward that $\{[\iiv]\}_{\iiv\in\mathcal{B}_\delta}$ is a partition of $\Sigma$. Since $\pi$ is surjective, the collection $\{\pi(B_\delta(\iiv))\}_{\iiv\in\mathcal{B}_\delta}$ gives a $\delta$-cover of $\Lambda$. Moreover, the OSC implies that $N_{\delta}(\Lambda)\approx \# \mathcal{B}_\delta$. As a result, it is enough to work with the finite sequences $\iiv\in\mathcal{B}_\delta$.

Since $L_\delta(\ii)$ depends on $\ii$, we adapt the method of types to handle sequences of different lengths simultaneously. Similarly as before, the \emph{type of $\ii\in\Sigma$ at scale $\delta$} is the empirical vector
\begin{equation*}
\tau_{\delta}(\ii) = \frac{1}{L_{\delta}(\ii)}\big( \#\{1\leq\ell\leq L_{\delta}(\ii):\, i_{\ell}=1\}, \ldots, \#\{1\leq\ell\leq L_{\delta}(\ii):\, i_{\ell}=N\}  \big).
\end{equation*}
The set of all possible types at scale $\delta$ is
\begin{equation*}
\mathcal{T}_\delta=\big\{\mathbf{p}: \text{ there exists } \iiv\in\mathcal{B}_\delta \text{ such that } \mathbf{p}=\tau_\delta(\iiv)\big\}
\end{equation*}
and the \emph{type class of} $\mathbf{p}\in \mathcal{T}_\delta$ is the set
\begin{equation*}
T_\delta(\mathbf{p}) = \big\{\iiv\in\mathcal{B}_\delta:\, \tau_\delta(\iiv)=\mathbf{p}\big\}.
\end{equation*} 
For fixed $\mathbf{p}\in\mathcal{T}_{\delta}$, observe that within its type class it follows from~\eqref{eq:03} that $L_{\delta}(\ii)\approx -\log \delta / \chi(\mathbf{p})$ for all $\ii$ such that $\ii\in[\iiv]$ for some $\iiv\in T_{\delta}(\mathbf{p})$. Thus, \eqref{eq:02} implies that
\begin{equation}\label{eq:04}
\left(\frac{-\log \delta}{\chi(\mathbf{p})}+1\right)^{-N} \delta^{-H(\mathbf{p})/\chi(\mathbf{p})} \lessapprox \#T_\delta(\mathbf{p}) \lessapprox \delta^{-H(\mathbf{p})/\chi(\mathbf{p})}.
\end{equation}
To bound $\#\mathcal{T}_\delta$ from above, note that $\log \delta / \log \lambda_{\min} \lessapprox L_{\delta}(\ii) \lessapprox \log \delta / \log \lambda_{\max}$. Then from~\eqref{eq:01} the following crude upper bound follows
\begin{equation}\label{eq:05}
\#\mathcal{T}_\delta \lessapprox \left(\frac{1}{\log \lambda_{\max}} - \frac{1}{\log \lambda_{\min}}\right)\log \delta \cdot \left( \frac{\log \delta}{\log \lambda_{\max}} +1 \right)^N.
\end{equation}
Let $\mathbf{p}^*_\delta\in\mathcal{T}_\delta$ denote the type for which $H(\mathbf{p}^*_\delta)/\chi(\mathbf{p}^*_\delta) = \max_{\mathbf{p}\in\mathcal{T}_\delta}H(\mathbf{p})/\chi(\mathbf{p})$. Then combining~\eqref{eq:04} and~\eqref{eq:05} with~\eqref{eq:00}, we obtain that
\begin{equation*}
\delta^{-H(\mathbf{p}^*_\delta)/\chi(\mathbf{p}^*_\delta)} \cdot O\big((-\log \delta)^{-N}\big) \lessapprox N_{\delta}(\Lambda) \lessapprox \delta^{-H(\mathbf{p}^*_\delta)/\chi(\mathbf{p}^*_\delta)} \cdot O\big((-\log \delta)^{N+1}\big).
\end{equation*} 
Since $\mathcal{T}_\delta$ becomes dense in $\mathcal{P}_N$ as $\delta\to 0$, moreover, $H(\mathbf{p})/\chi(\mathbf{p})$ is continuous in $\mathbf{p}$, we conclude that $\mathbf{p}^*_\delta\to \mathbf{p}^*$ as $\delta\to 0$, where $\mathbf{p}^*\in\mathcal{P}_N$ maximises $H(\mathbf{p})/\chi(\mathbf{p})$ over all $\mathbf{p}\in\mathcal{P}_N$. Hence, $\dim_{\mathrm B}\Lambda = H(\mathbf{p}^*)/\chi(\mathbf{p}^*)$. To finish, a standard use of the Lagrange multipliers shows that $\mathbf{p}^*=(\lambda_1^s,\ldots,\lambda_N^s)$. As a result, $\dim_{\mathrm B}\Lambda=s$ as claimed. Thus, $\mathbf{p}^*$ is the dominant box counting type and $T_\delta(\mathbf{p}^*)$ is the dominant box counting class in this case. 

\begin{remark}\label{remark1}
If $\Lambda$ is a \emph{homogeneous} self-similar set, i.e. $\lambda_i\equiv \lambda$, then $\chi(\mathbf{p}) = -\log \lambda$ for any $\mathbf{p}$. Hence, the dominant box counting type is the uniform measure $\mathbf{p}^*=(1/N,\ldots,1/N)$ since it maximises $H(\mathbf{p})$ with value $\log N$, which implies that $\dim_{\mathrm B}\Lambda=s=\log N/(-\log \lambda)$.
\end{remark}

\section{Main results about self-affine sponges}\label{sec:2}

The main application of the method of types in this paper is to determine the box dimension of self-affine sponges in $\R^d$ of Gatzouras--Lalley and of Bara\'nski type. A \emph{self-affine set} is the attractor of an IFS in which all maps have the form $S_i(\underline{x}) = A_i \underline{x}+\underline{t}_i$, where $A_i$ is a contracting non-singular $d\times d$ matrix and $\underline{t}_i\in\R^d$.

Loosely speaking, sponges are referred to as higher dimensional analogues of self-affine carpet-like constructions on the plane. The key features of these constructions is their excessive alignment of cylinders and defining diagonal matrices. The significance of these carpets is that they provide explicit examples for which the various notions of dimension are different. They are part of a very small family of exceptions, since the box and Hausdorff dimensions of self-affine sets coincide in a `typical' sense~\cite{falconer_1988,falconer_1992} in $\R^d$ and also in a more explicit sense~\cite{BaranyHochmanRapaport,hochmanRapaport2019arxiv} in $\R^2$.

Self-affine carpets were first studied independently by Bedford~\cite{Bedford84_phd} and  McMullen~\cite{mcmullen84}. Their construction was generalised by Gatzouras and Lalley~\cite{GatzourasLalley92} and later by Bara\'nski~\cite{BARANSKIcarpet_2007}. The various dimensions of these basic models are well understood. Most of these results have been generalised in different directions on the plane to constructions with overlaps~\cite{fraser_shmerkin_2016,KolSimon_TriagGL2019,pardo-simon}, to constructions using lower triangular matrices~\cite{BaranskiTriag_2008,KolSimon_TriagGL2019} or to more general `box-like' constructions~\cite{FengWang2005,Fraser12Boxlike}. Figure~\ref{fig:CarpetPicture} shows different carpet-like constructions with increasing complexity. In each case, the shaded rectangles or parallelograms are the images of $[0,1]^2$ under the maps of the defining IFS. The attractor is obtained by repeatedly applying the maps to the remaining shaded areas ad infinitum.  

\begin{figure}[h]
	\centering
	\includegraphics[width=0.95\textwidth]{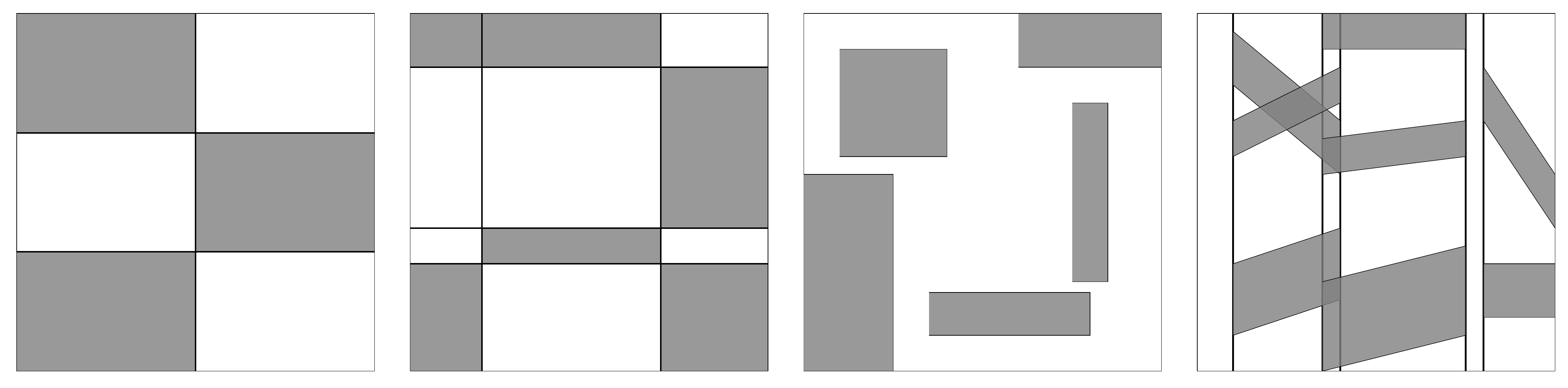}
	\caption{Different carpet-like constructions. From left to right: Bedford--McMullen carpet, Bara\'nski carpet, box-like construction, triangular Gatzaouras--Lalley carpet with overlaps.}
	\label{fig:CarpetPicture}
\end{figure}

Much less is known, however, about the dimension theory of self-affine sponges. In the simplest case of a Bedford--McMullen (also referred to as Sierpi\'nski) sponge, its Hausdorff and box dimensions were obtained by Kenyon and Peres~\cite{KenyonPeres_ETDS96}, while its lower and Assouad dimensions by Fraser and Howroyd~\cite{FraserHowroyd_AnnAcadSciFennMath17}. Feng and Hu~\cite{FengHu09} relaxed the separation condition in case of the Hausdorff and box dimension. Perhaps the paper with the most impact is due to Das and Simmons~\cite{das2017hausdorff}, who by calculating the Hausdorff dimension of Gatzouras--Lalley and Bara\'nski sponges gave the first example of a set which does not have a shift invariant measure of maximal Hausdorff dimension, thus resolving a long standing open problem in dynamical systems. The closest related result is a recent work of Fraser and Jurga~\cite{fraserJurga2019SpongeboxArXiv}, where they obtain results about the box dimension of certain sponges in $\R^3$ generated by generalised permutation matrices, which contain the Gatzouras--Lalley sponges but not the Bara\'nski type. We continue with the formal definitions and state our main results. 

\subsection{Gatzouras--Lalley sponges}\label{sec:21}

The definition is slightly technical and needs some notation. We begin by defining the collection of index sets $\mathcal{I}_1,\mathcal{I}_2,\ldots,\mathcal{I}_d$ as follows:
\begin{enumerate}
\item Fix an integer $N\geq 2$ and let $\mathcal{I}_1:=\{1,\ldots,N\}$;
\item For each $i_1\in\mathcal{I}_1$ fix $N(i_1)\in\N$ and let $\mathcal{I}(i_1):=\{1,\ldots,N(i_1)\}$, moreover, 
\begin{equation*}
\mathcal{I}_2:= \bigcup_{i_1 \in \mathcal{I}_{1}}\; \bigcup_{i_2 \in \mathcal{I}(i_1)}(i_1, i_2);
\end{equation*} 
\item Continue inductively for $3\leq n\leq d$: given $\mathcal{I}_{n-1}$, fix $N(\iu)\in\N$ for each $\iu\in\mathcal{I}_{n-1}$. Let $\mathcal{I}(\iu):=\{1,\ldots,N(\iu)\}$ and finally
\begin{equation*}
\mathcal{I}_{n}:=\bigcup_{\iu \in \mathcal{I}_{n-1}}\; \bigcup_{i_{n} \in \mathcal{I}(\iu)}(\iu, i_{n}).
\end{equation*}
\end{enumerate}
We extensively use projections. To denote the projection of $\iu=(i_1,\ldots,i_n)\in\mathcal{I}_n$ (where $1\leq n\leq d$) to its first $\ell\leq n$ coordinates, we use the notation
\begin{equation*}
\iu^{(\ell)} := (i_1,\ldots,i_\ell).
\end{equation*}
The same notation is extended to vectors $\underline{v}$ and subsets $V$ of $\R^n$: $\underline{v}^{(\ell)} = (v_1,\ldots,v_\ell)$ and $V^{(\ell)} = \bigcup_{\underline{v}\in V}\underline{v}^{(\ell)}$. 

We can now introduce the IFSs $\mathcal{S}_1,\ldots,\mathcal{S}_d$, where $\mathcal{S}_n=\{S_{\iu}: \R^n\to\R^n\}_{\iu\in\mathcal{I}_n}$ is defined as
\begin{equation*}
S_{\iu}(\underline{x}) = A_{\iu} \underline{x} + \underline{t}_{\iu} = 
\begin{bmatrix}
\lambda(\iu^{(1)}) & & 0\\
 &\ddots & \\
0 & & \lambda(\iu^{(n)}) \\
\end{bmatrix} \cdot \underline{x} + \begin{bmatrix}
t(\iu^{(1)}) \\
\vdots \\
t(\iu^{(n)}) \\
\end{bmatrix}.
\end{equation*}
The $\ell$-th coordinate of $S_{\iu}(\underline{x})$ only depends on the first $\ell$ coordinates of $\iu$. We assume that $0<\lambda(\iu)<1$ for every $1\leq n\leq d$ and $\iu\in\mathcal{I}_n$. Without loss of generality we assume that the $t_{\iu}$ are chosen so that $S_{\iu}([0,1]^n)\subset [0,1]^n$. The attractor of $\mathcal{S}_n$ is the unique, non-empty compact set $\Lambda_{n}$ satisfying the relation
\begin{equation*}
\Lambda_{n}=\bigcup_{\iu \in \mathcal{I}_{n}} S_{\iu}(\Lambda_{n}).
\end{equation*}

\begin{definition}
The attractor $\Lambda_{d}$ of $\mathcal{S}_d$ is a \emph{Gatzouras--Lalley (GL) sponge} in $\R^d$ if 
\begin{equation}\label{eq:20}
S_{\iu}((0,1)^d)\cap S_{\underline{j}}((0,1)^d)=\emptyset \;\text{ for every } \iu\neq\underline{j}\in\mathcal{I}_d
\end{equation}
and 
\begin{equation}\label{eq:21}
0<\lambda(\iu^{(d)})<\lambda(\iu^{(d-1)})<\ldots<\lambda(\iu^{(1)})<1 \;\text{ for every } \iu\in\mathcal{I}_d.
\end{equation}
We call condition~\eqref{eq:20} the \emph{cuboidal open set condition (COSC)} and condition~\eqref{eq:21} is the \emph{coordinate ordering condition}.
\end{definition}

\begin{remark}\
\begin{enumerate}
\item Observe that $\Lambda_{n}^{(\ell)} = \Lambda_{\ell}$ for any $\ell\leq n\leq d$. In addition, if $\Lambda_{d}$ is a GL-sponge in $\R^d$, then $\Lambda_{d}^{(n)}$ is a GL-sponge in $\R^n$ for every $1\leq n\leq d$.
\item The COSC, introduced in~\cite{fraserJurga2019SpongeboxArXiv}, is the higher dimensional analogue of the rectangular OSC defined in~\cite{FengWang2005}. The name for~\eqref{eq:21} is taken from~\cite{das2017hausdorff}.
\item Condition~\eqref{eq:21} could be assumed for any permutation of the coordinates, as long as the permutation is the same for all $\iu\in\mathcal{I}_d$. We chose this to simplify notation.
\end{enumerate}
\end{remark}

\begin{theorem}\label{thm:main}
Let $\Lambda_{d}$ be a Gatzouras--Lalley sponge in $\R^d$. Then
\begin{equation*}
\dim_{\mathrm B} \Lambda_{d} = s_d,
\end{equation*}
where the numbers $s_1\leq s_2\leq \ldots \leq s_d$ are defined as the unique solutions to the equations
\begin{equation}\label{eq:22}
\sum_{i\in\mathcal{I}_1} (\lambda(i))^{s_1}=1 \;\text{ and }\; \sum_{\iu\in\mathcal{I}_{n}} \big(\lambda(\iu^{(1)})\big)^{s_1}\cdot \prod_{\ell=2}^n \big(\lambda(\iu^{(\ell)})\big)^{s_{\ell}-s_{\ell-1}} =1 \text{ for } n=2,\ldots,d.
\end{equation}
\end{theorem}
The equations in~\eqref{eq:22} naturally define probability vectors $\mathbf{p}_1^*,\ldots,\mathbf{p}_d^*$. These vectors define a multi-dimensional type class, see Section~\ref{sec:32}, and the proof reveals that this type class is the dominant box counting class.

The theorem in two dimensions was first proved by Gatzouras and Lalley~\cite{GatzourasLalley92}, and  for $d=3$ it follows from a more general result of Fraser and Jurga~\cite{fraserJurga2019SpongeboxArXiv}. Their arguments are completely different and they are completely different from the proof presented here.

Besides the closed form for $s_d$ given by~\eqref{eq:22}, we also obtain a variational formula, see Proposition~\ref{prop:1} and Lemma~\ref{lem:1} for details. In particular, in two dimensions,
\begin{equation}\label{eq:23}
\dim_{\mathrm B}\Lambda_2 = \max_{\mathbf{p}_1\in\mathcal{P}_{\mathcal{I}_1},\, \mathbf{p}_2\in\mathcal{P}_{\mathcal{I}_2}}\; \frac{H(\mathbf{p}_2)}{\chi_2(\mathbf{p}_2)}  +
\left( 1- \frac{\chi_1(\mathbf{p}_2)}{\chi_2(\mathbf{p}_2)} \right) \frac{H(\mathbf{p}_1)}{\chi_1(\mathbf{p}_1)},
\end{equation}
where $\mathcal{P}_{\mathcal{I}_1}$ and $\mathcal{P}_{\mathcal{I}_2}$ denote the set of probability vectors on $\mathcal{I}_1$ and $\mathcal{I}_2$, respectively, and the Lyapunov exponents are $\chi_n(\mathbf{p}_m)= -\sum_{\iu\in\mathcal{I}_{m}} p_m(\iu) \log \lambda(\iu^{(n)})$ for $1\leq n\leq m\leq 2$. The formula resembles the Ledrappier--Young formula for Hausdorff dimension, see Section~\ref{sec:5} for a detailed discussion. For the three-dimensional analogue of this formula see~\eqref{eq:40}.

\subsection{Bara\'nski sponges}\label{sec:22}

The notation is slightly simpler in this case. For $1\leq n\leq d$, the index set $\mathcal{I}_n:=\{1,\ldots,N_n\}$ defines the base IFS $\mathcal{F}_n$ in coordinate $n$ by
\begin{equation*}
\mathcal{F}_n := \big\{ f_{n,i}(x) = \lambda_n(i)\cdot x + t_{n,i}\big\}_{i\in\mathcal{I}_n}, \;\text{ where } t_{n,i} = \sum_{\ell=1}^{i-1} \lambda_n(\ell).
\end{equation*}
The choice of $t_{n,i}$ implies that each $\mathcal{F}_n$ satisfies the OSC~\eqref{eq:06} with $V=(0,1)$. The alphabet is a subset $\mathcal{I}\subseteq \prod_{n=1}^d \mathcal{I}_n$ and an element of it is $\iu=(i_1,\ldots,i_d)$. For a subset $D\subset\{1,\ldots,d\}$ let $\Pi(\iu;D):=(i_{\ell})_{\ell\in D}$, i.e. the coordinates of $\iu$ whose indices belong to $D$. Similarly, $\Pi(\mathcal{I};D):=\{\Pi(\iu;D):\, \iu\in\mathcal{I}\}$.   
\begin{definition}\label{def:Baranski}
The IFS $\mathcal{S}=\{S_{\iu}\}_{\iu\in\mathcal{I}}$ is of \emph{Bara\'nski type} if
\begin{equation*}
S_{\iu}(\underline{x}) = \big( f_{1,i_1}(x_1),\ldots,f_{d,i_d}(x_d) \big).
\end{equation*}
The attractor $\Lambda = \bigcup_{\iu \in \mathcal{I}}S_{\iu}(\Lambda)$ is a \emph{Bara\'nski sponge} in $\R^d$.
\end{definition}

To state the result in this case, we let $\mathrm{Sym}(\{1,\ldots, d\})$ denote the symmetric group on the set of coordinates $\{1,\ldots,d\}$ and denote a permutation by
\begin{equation}\label{eq:24}
\sigma = \begin{pmatrix}
1 & 2 & \cdots & d \\
\sigma_1 & \sigma_2 & \cdots & \sigma_d
\end{pmatrix}
\in \mathrm{Sym}(\{1,\ldots, d\}).
\end{equation}

\begin{theorem}\label{thm:main2}
Let $\Lambda_{d}$ be a Bara\'nski sponge in $\R^d$. Then
\begin{equation*}
\dim_{\mathrm B} \Lambda_{d} = \max_{\sigma\in\mathrm{Sym}(\{1,\ldots, d\})} s_d(\sigma),
\end{equation*}
where for a fixed $\sigma=(\sigma_1,\ldots,\sigma_d)\in\mathrm{Sym}(\{1,\ldots d\})$ the numbers $s_1(\sigma)\leq s_2(\sigma)\leq \ldots \leq s_d(\sigma)$ are defined as the unique solutions to the equations
\begin{align*}
\sum_{i\in\Pi(\mathcal{I};\{\sigma_1\})} (\lambda_{\sigma_1}(i))^{s_1(\sigma)}&=1 \;\text{ and }\; \\ \sum_{(i_1\ldots,i_n)\in\Pi(\mathcal{I};\{\sigma_1,\ldots,\sigma_n\})} \big(\lambda_{\sigma_1}(i_1)\big)^{s_1(\sigma_1)}\cdot \prod_{\ell=2}^n \big(\lambda_{\sigma_{\ell}}(i_{\ell})\big)^{s_{\ell}(\sigma)-s_{\ell-1}(\sigma)} &=1 \;\text{ for } n=2,\ldots,d.
\end{align*}
\end{theorem}
Essentially the theorem states that for every possible ordering of the coordinates, one has to calculate the numbers $s_1(\sigma),\ldots,s_d(\sigma)$ like in the GL case and then take a maximum. The reason why all orderings are considered is because the coordinate ordering condition~\eqref{eq:21} is not assumed for Bara\'nski sponges. The theorem in two dimensions was first proved in~\cite{BARANSKIcarpet_2007}, but the proof is different from the one presented here.
\begin{remark}
The packing dimension of every Gatzouras--Lalley or Bara\'nski sponge is equal to its box dimension, since $\Lambda$ is compact and every open set intersecting $\Lambda$ contains a bi-Lipschitz image of $\Lambda$, see~\cite[Corollary 3.9]{FalconerBook}. 
\end{remark}

\section{Preliminaries}\label{sec:3}

This section introduces approximate cubes and multi-dimensional types. Here we concentrate on Gatzouras--Lalley  sponges. The slight modifications for Bara\'nski sponges are discussed in Section~\ref{sec:41}.
 
\subsection{Approximate cubes}\label{sec:31}

The natural generalization of approximate squares used extensively in covering arguments for self-affine carpets on the plane are approximate cubes in higher dimensions. The \emph{$\delta$-stopping of $\ii\in\Sigma$ in the $n$-th coordinate} (for $n=1,\ldots,d$) is the unique integer $L_{\delta}(\ii,n)$ such that
\begin{equation}\label{eq:34}
\prod_{\ell=1}^{L_{\delta}(\ii,n)} \lambda\big(\iu_\ell^{(n)}\big) \leq \delta < \prod_{\ell=1}^{L_{\delta}(\ii,n)-1}\lambda\big(\iu_\ell^{(n)}\big).
\end{equation}
Also let $L_{\delta}(\ii,d+1):=0$. The \emph{symbolic $\delta$-approximate cube} containing $\ii\in\Sigma$ is
\begin{equation*}
B_\delta(\ii) = \big\{ \jj\in\Sigma:\, \ii^{(n)}|L_{\delta}(\ii,n) = \jj^{(n)}|L_{\delta}(\ii,n)\; \text{ for every } n=1,\ldots,d\big\}.
\end{equation*}
It is easy to see that for $\ii\neq\jj\in\Sigma$, either $B_\delta(\ii)=B_\delta(\jj)$ or $B_\delta(\ii)\cap B_\delta(\jj)=\emptyset$. Hence, the set of approximate cubes $\mathcal{B}_\delta$ defines a partition of $\Sigma$. To make a distinction, for each element $B_\delta(\ii)$ of the partition $\mathcal{B}_\delta$ we choose an $\iih\in B_\delta(\ii)$ to `represent' it and write $B_\delta(\iih)\in\mathcal{B}_\delta$. Since we assume the COSC~\eqref{eq:20}, the image of two elements $B_\delta(\iih)\neq B_\delta(\jjh)\in\mathcal{B}_\delta$ by the natural projection $\pi$ on $\Lambda$ can only intersect on their boundary, so we obtain a cover of $\Lambda$ for which $\#\mathcal{B}_\delta=N_{\delta}(\Lambda)$. As a result, it is enough to work with the set of symbolic approximate cubes $\mathcal{B}_\delta$.

\subsection{Multidimensional types}\label{sec:32}

In order to introduce multidimensional types, first observe that every approximate cube $B_\delta(\ii)$ can be uniquely identified with the finite sequence 
\begin{equation}\label{eq:30}
\big\{ \underbrace{\iu_1^{(d)},\ldots,\iu_{L_{\delta}(\ii,d)}^{(d)}}_{\in(\mathcal{I}_d)^{L_{\delta}(\ii,d)}}\,;\,
\underbrace{\iu_{L_{\delta}(\ii,d)+1}^{(d-1)},\ldots,\iu_{L_{\delta}(\ii,d-1)}^{(d-1)}}_{\in(\mathcal{I}_{d-1})^{L_{\delta}(\ii,d-1)-L_{\delta}(\ii,d)}}
\,;\,\ldots\,;\, 
\underbrace{\iu_{L_{\delta}(\ii,2)+1}^{(1)},\ldots,\iu_{L_{\delta}(\ii,1)}^{(1)}}_{\in(\mathcal{I}_1)^{L_{\delta}(\ii,1)-L_{\delta}(\ii,2)}}\big\}.
\end{equation}
This identification is indeed one-to-one because the coordinate ordering condition~\eqref{eq:21} implies that $L_{\delta}(\ii,d)<L_{\delta}(\ii,d-1)<\ldots<L_{\delta}(\ii,1)$ for every $\ii\in\Sigma$.

In this setting, the \emph{type of $\ii\in\Sigma$ at scale $\delta$} is the $\#\mathcal{I}_d+\#\mathcal{I}_{d-1}+\ldots+\#\mathcal{I}_1$ dimensional empirical vector
\begin{equation*}
\tau_{\delta}(\ii) = \big( \tau_{\delta}(\ii,d)\,;\, \tau_{\delta}(\ii,d-1)\,;\,\ldots \,;\, \tau_{\delta}(\ii,1) \big),
\end{equation*}
where for $1\leq n\leq d$
\begin{equation*}
\tau_{\delta}(\ii,n) = \frac{1}{L_{\delta}(\ii,n)-L_{\delta}(\ii,n+1)} \Big( \#\big\{L_{\delta}(\ii,n+1)+1\leq\ell\leq L_{\delta}(\ii,n):\, \iu_{\ell}^{(n)}=\underline{j}\big\}  \Big)_{\underline{j}\in\mathcal{I}_n} .
\end{equation*}
Note that $\tau_{\delta}(\ii,n)$ is an $\#\mathcal{I}_n$ dimensional probability vector. The set of all possible types at scale $\delta$ is
\begin{equation*}
\mathcal{T}_\delta=\big\{\mathbf{P}=(\mathbf{p}_d;\mathbf{p}_{d-1};\ldots; \mathbf{p}_1): \text{ there exists } B_\delta(\iih)\in\mathcal{B}_\delta \text{ such that } \mathbf{P}=\tau_\delta(\iih)\big\},
\end{equation*}
and the \emph{type class of} $\mathbf{P}\in \mathcal{T}_\delta$ is the set
\begin{equation*}
T_\delta(\mathbf{P}) = \big\{B_\delta(\iih)\in\mathcal{B}_\delta:\, \tau_\delta(\iih)=\mathbf{P}\big\}.
\end{equation*}
Let $\mathbf{p}_m$ be a probability vector on $\mathcal{I}_m$. For $1\leq n\leq m$, we denote the \emph{Lyapunov exponent} by
\begin{equation*}
\chi_n(\mathbf{p}_m) := -\sum_{\iu\in\mathcal{I}_{m}} p_m(\iu) \log \lambda(\iu^{(n)}).
\end{equation*}

\begin{lemma}\label{lem:2}
Fix a type $\mathbf{P}=(\mathbf{p}_d;\mathbf{p}_{d-1};\ldots; \mathbf{p}_1)\in\mathcal{T}_\delta$. For every $1\leq n\leq d$ there exists a constant $C_n^{(d)}(\mathbf{P})$ depending on $\mathbf{P}$ only through $\chi_\ell(\mathbf{p}_m)$ for $n\leq \ell\leq m\leq d$ such that
\begin{equation*}
L_{\delta}(\iih,n)-L_{\delta}(\iih,n+1) \approx -C_n^{(d)}(\mathbf{P}) \cdot\log \delta, \;\text{ where } \iih\in\Sigma \text{ is such that } \tau_\delta(\iih)=\mathbf{P}.
\end{equation*}
Moreover, $\sum_{m=n+1}^{d}C_m^{(d)}(\mathbf{P})\cdot \chi_{n+1}(\mathbf{p}_m)=1$ for every $1\leq n\leq d-1$. 
\end{lemma}
\begin{proof}
For each $1\leq n\leq d$ and fixed $\mathbf{P}=(\mathbf{p}_d;\mathbf{p}_{d-1};\ldots; \mathbf{p}_1)\in\mathcal{T}_\delta$, it follows that
\begin{align*}
\log \delta &\approx  \sum_{m=n}^{d}\; \big(L_{\delta}(\iih,m)-L_{\delta}(\iih,m+1)\big) \frac{1}{L_{\delta}(\iih,m)-L_{\delta}(\iih,m+1)}\sum_{\ell=L_{\delta}(\iih,m+1)+1}^{L_{\delta}(\iih,m)} \log \lambda\big(\iu_\ell^{(n)}\big) \\
&=  -\sum_{m=n}^{d} \big(L_{\delta}(\iih,m)-L_{\delta}(\iih,m+1)\big)\cdot \chi_n(\mathbf{p}_m).
\end{align*}
In particular, for $n=d$ (recall $L_{\delta}(\iih,d+1)=0$ by definition), $\log \delta \approx -L_{\delta}(\iih,d)\cdot \chi_d(\mathbf{p}_d)$, giving $C_d^{(d)}(\mathbf{P})=1/\chi_d(\mathbf{p}_d)$. In the next step for $n=d-1$,
\begin{equation*}
L_{\delta}(\iih,d-1)-L_{\delta}(\iih,d) \approx \frac{-\big( \log \delta + L_{\delta}(\iih,d)\cdot \chi_{d-1}(\mathbf{p}_d) \big)}{\chi_{d-1}(\mathbf{p}_{d-1})} = \underbrace{\left( 1-\frac{\chi_{d-1}(\mathbf{p}_d)}{\chi_{d}(\mathbf{p}_d)} \right)\frac{-\log \delta}{\chi_{d-1}(\mathbf{p}_{d-1})}}_{=:-C_{d-1}^{(d)}(\mathbf{P})\cdot\log \delta}.
\end{equation*}
The argument continues by induction as $n$ decreases further. After rearranging,
\begin{align}
L_{\delta}(\iih,n)-L_{\delta}(\iih,n+1) &\approx \frac{-1}{\chi_n(\mathbf{p}_n)}\bigg( \log \delta + \sum_{m=n+1}^{d} \big(L_{\delta}(\iih,m)-L_{\delta}(\iih,m+1)\big)\cdot \chi_n(\mathbf{p}_m) \bigg) \nonumber \\
&= \!\bigg(\! 1-\!\sum_{m=n+1}^{d} C_m^{(d)}(\mathbf{P})\cdot \chi_n(\mathbf{p}_m)\! \bigg) \frac{-\log \delta}{\chi_n(\mathbf{p}_n)} =: -C_n^{(d)}(\mathbf{P}) \cdot\log \delta. \label{eq:35}
\end{align}
The final assertion follows simply by applying the definition of $C_{n+1}^{(d)}(\mathbf{P})$:
\begin{multline*}
\sum_{m=n+1}^{d}C_m^{(d)}(\mathbf{P})\cdot \chi_{n+1}(\mathbf{p}_m) = C_{n+1}^{(d)}(\mathbf{P})\cdot \chi_{n+1}(\mathbf{p}_{n+1}) + \sum_{m=n+2}^{d}C_m^{(d)}(\mathbf{P})\cdot \chi_{n+1}(\mathbf{p}_m)  \\
= \bigg( 1-\sum_{m=n+2}^{d}C_m^{(d)}(\mathbf{P})\cdot \chi_{n+1}(\mathbf{p}_m) \bigg) \frac{\chi_{n+1}(\mathbf{p}_{n+1})}{\chi_{n+1}(\mathbf{p}_{n+1})} + \sum_{m=n+2}^{d}C_m^{(d)}(\mathbf{P})\cdot \chi_{n+1}(\mathbf{p}_m) =1.
\end{multline*} 
\end{proof}

\section{Proof of Theorems~\ref{thm:main} and \ref{thm:main2}}\label{sec:4}

We begin with the proof of Theorem~\ref{thm:main} and then show what adjustments need to be made to the argument to prove Theorem~\ref{thm:main2}.

\subsection{Proof of Theorem~\ref{thm:main}}\label{sec:40}

Let $\mathcal{P}_{1,\ldots,d}$ denote the set of all $\#\mathcal{I}_d+\#\mathcal{I}_{d-1}+\ldots+\#\mathcal{I}_1$ dimensional vectors $\mathbf{P}=(\mathbf{p}_d;\mathbf{p}_{d-1};\ldots; \mathbf{p}_1)$, where each $\mathbf{p}_n$ is a probability vector on $\mathcal{I}_n$. We are ready to state our variational formula for $\dim_{\mathrm B}\Lambda_d$.
\begin{prop}\label{prop:1}
Let $\Lambda_{d}$ be a Gatzouras--Lalley sponge in $\R^d$. Then
\begin{equation}\label{eq:33}
\dim_{\mathrm B} \Lambda_{d} = \max_{\mathbf{P}\in \mathcal{P}_{1,\ldots,d}}\; \sum_{n=1}^d C_n^{(d)}(\mathbf{P})\cdot H(\mathbf{p}_n),
\end{equation}
where $C_n^{(d)}(\mathbf{P})$ is defined in Lemma~\ref{lem:2}. 
\end{prop}

\begin{proof}
The main step is to apply the method of types to the multi-dimensional type $\mathbf{P}\in \mathcal{T}_\delta$. For any type $\mathbf{P}\in \mathcal{T}_\delta$, we repeatedly use~\eqref{eq:02} for each $\mathbf{p}_d,\mathbf{p}_{d-1},\ldots, \mathbf{p}_1$ to get that
\begin{multline*}
\mathrm{exp}\left[ \sum_{n=1}^d (L_{\delta}(\iih,n)-L_{\delta}(\iih,n+1))H(\mathbf{p}_n)\right] \cdot \prod_{n=1}^d \big(L_{\delta}(\iih,n)-L_{\delta}(\iih,n+1)+1\big)^{-\#\mathcal{I}_n}  \\
\lessapprox \#T_\delta(\mathbf{P})  \lessapprox 
\mathrm{exp}\left[ \sum_{n=1}^d (L_{\delta}(\iih,n)-L_{\delta}(\iih,n+1))H(\mathbf{p}_n)\right],
\end{multline*}
where $\iih\in\Sigma$ is such that $\tau_\delta(\iih)=\mathbf{P}$. From Lemma~\ref{lem:2} it follows that
\begin{equation}\label{eq:31}
\delta^{-\sum_{n=1}^d C_n^{(d)}(\mathbf{P})\cdot H(\mathbf{p}_n)}\cdot ( -\log \delta)^{-\sum_{n=1}^d \#\mathcal{I}_n}\lessapprox \#T_\delta(\mathbf{P})  \lessapprox \delta^{-\sum_{n=1}^d C_n^{(d)}(\mathbf{P})\cdot H(\mathbf{p}_n)}.
\end{equation}
On the other hand, we can give a crude upper bound for $\#\mathcal{T}_{\delta}$ in a similar fashion by repeating the argument in the self-similar case~\eqref{eq:05} for each coordinate $1\leq n\leq d$:
\begin{equation}\label{eq:32}
\# \mathcal{T}_\delta \leq \prod_{n=1}^d \big(L_{\delta}(\iih,n)-L_{\delta}(\iih,n+1)+1\big)^{\#\mathcal{I}_n}\cdot \max_{B_\delta(\iih)\in\mathcal{B}_\delta}(L_{\delta}(\iih,n)-L_{\delta}(\iih,n+1)).
\end{equation}
Again by Lemma~\ref{lem:2}, the right hand side in~\eqref{eq:32} is $o(\delta^{-1})$.
It follows from~\eqref{eq:31} that the dominant box counting type $\mathbf{P}^*_{\delta}=(\mathbf{p}_{\delta,d}^*;\mathbf{p}_{\delta,d-1}^*;\ldots; \mathbf{p}_{\delta,1}^*)\in\mathcal{T}_\delta$ maximises the expression $\sum_{n=1}^d C_n^{(d)}(\mathbf{P})\cdot H(\mathbf{p}_n)$. Thus, combining \eqref{eq:31} and \eqref{eq:32} with~\eqref{eq:00} implies that
\begin{equation*}
(-\log \delta)\cdot \sum_{n=1}^d  C_n^{(d)}(\mathbf{P}^*_\delta)\cdot H(\mathbf{p}^*_{n,\delta})- \varepsilon  \lessapprox \log N_{\delta}(\Lambda) \lessapprox (-\log \delta)\cdot\sum_{n=1}^d  C_n^{(d)}(\mathbf{P}^*_\delta)\cdot H(\mathbf{p}^*_{n,\delta}) + \varepsilon,
\end{equation*}
where the error term $\varepsilon=o(-\log \delta)$. As a result, we obtain the variational formula
\begin{equation*}
\dim_{\mathrm B} \Lambda_{d} = \lim_{\delta\to 0}\; \sum_{n=1}^d  C_n^{(d)}(\mathbf{P}^*_\delta)\cdot H(\mathbf{p}^*_{n,\delta}).
\end{equation*}
As $\delta\to 0$, the set of types $\mathcal{T}_\delta$ becomes dense in the set $\mathcal{P}_{1,\ldots,d}$. The compactness of $\mathcal{P}_{1,\ldots,d}$ and the continuity of $C_n^{(d)}(\mathbf{P})$ and $H(\mathbf{p}_n)$ implies that the dominant box counting type $\mathbf{P}^*_{\delta}$ tends to the limiting dominant type $\mathbf{P}^*=(\mathbf{p}_{d}^*;\mathbf{p}_{d-1}^*;\ldots; \mathbf{p}_{1}^*)\in\mathcal{P}_{1,\ldots,d}$, which satisfies
\begin{equation}\label{eq:41}
\sum_{n=1}^d C_n^{(d)}(\mathbf{P}^*)\cdot H(\mathbf{p}^*_n) = \max_{\mathbf{P}\in \mathcal{P}_{1,\ldots,d}}\; \sum_{n=1}^d C_n^{(d)}(\mathbf{P})\cdot H(\mathbf{p}_n) = \dim_{\mathrm B} \Lambda_{d}.
\end{equation}
\end{proof}

\begin{remark}
It is possible to express the formula in~\eqref{eq:33} in terms of Lyapunov exponents. In particular, for $d=2$ one obtains the formula already presented in~\eqref{eq:23}. For $d=3$, slightly more work shows that the expression to be maximised is
\begin{equation}\label{eq:40}
\frac{H(\mathbf{p}_3)}{\chi_3(\mathbf{p}_3)} + \left( 1-\frac{\chi_2(\mathbf{p}_3)}{\chi_3(\mathbf{p}_3)} \right)\frac{H(\mathbf{p}_2)}{\chi_2(\mathbf{p}_2)} + \left[ 1-\frac{\chi_1(\mathbf{p}_3)}{\chi_3(\mathbf{p}_3)} - \left( 1-\frac{\chi_2(\mathbf{p}_3)}{\chi_3(\mathbf{p}_3)} \right) \frac{\chi_1(\mathbf{p}_2)}{\chi_2(\mathbf{p}_2)}\right]\frac{H(\mathbf{p}_1)}{\chi_1(\mathbf{p}_1)}.
\end{equation}
We think of it as a Ledrappier--Young like formula for the box dimension, see Section~\ref{sec:5} for further discussion. For $d>3$ the calculations get increasingly involved and cumbersome. 
\end{remark}
The next lemma characterises the limiting dominant type $\mathbf{P}^*=(\mathbf{p}_{d}^*;\mathbf{p}_{d-1}^*;\ldots; \mathbf{p}_{1}^*)$ and concludes the proof of Theorem~\ref{thm:main}. 

\begin{lemma}\label{lem:1}
Let $\Lambda_{d}$ be a Gatzouras--Lalley sponge in $\R^d$. Then the maximum in~\eqref{eq:33} is uniquely attained by the type $\mathbf{P}^*=(\mathbf{p}_d^*;\mathbf{p}_{d-1}^*;\ldots; \mathbf{p}_1^*)$, where the probability vectors $\mathbf{p}_n^*=(p_n^*(\iu))_{\iu\in\mathcal{I}_n}$ are defined by
\begin{equation*}
p_1^*(i)=(\lambda(i))^{s_1} \;\;\text{ and }\;\; p_n^*(\iu)=\big(\lambda(\iu^{(1)})\big)^{s_1}\cdot \prod_{\ell=2}^n \big(\lambda(\iu^{(\ell)})\big)^{s_{\ell}-s_{\ell-1}} \text{ for } n=2,\ldots,n,
\end{equation*}
where $s_1\leq s_2\leq \ldots\leq s_d$ were introduced in~\eqref{eq:22}. Moreover, $\dim_{\mathrm B}\Lambda_d =  s_d.$
\end{lemma}
\begin{proof}
We start by showing that $\dim_{\mathrm B}\Lambda_d =  s_d$. Immediate calculations yield that
\begin{equation*}
H(\mathbf{p}_n^*) = \sum_{\ell=1}^n (s_{\ell}-s_{\ell-1})\cdot \chi_\ell(\mathbf{p}_n^*),
\end{equation*}
where we define $s_0:=0$. Substituting this into~\eqref{eq:41}, we see that $\dim_{\mathrm B}\Lambda_d$ equals
\begin{align*}
&\sum_{n=1}^d C_n^{(d)}(\mathbf{P}^*)\cdot \sum_{\ell=1}^n (s_{\ell}-s_{\ell-1})\cdot \chi_\ell(\mathbf{p}_n^*) = \sum_{\ell=1}^d (s_{\ell}-s_{\ell-1}) \sum_{n=\ell}^d C_n^{(d)}(\mathbf{P}^*)\cdot \chi_\ell(\mathbf{p}_n^*) \\
&\phantom{\sum_{n=1}^d}= s_d\cdot C_d^{(d)}(\mathbf{P}^*)\cdot \chi_{d}(\mathbf{p}_d^*) +\! \sum_{\ell=1}^{d-1} s_{\ell} \Big(\! \underbrace{ C_{\ell}^{(d)}(\mathbf{P}^*)\cdot \chi_\ell(\mathbf{p}^*_\ell) +\!\! \sum_{n=\ell+1}^d C_n^{(d)}(\mathbf{P}^*) \big( \chi_{\ell}(\mathbf{p}_n^*) - \chi_{\ell+1}(\mathbf{p}_n^*) \big) }_{=:A_\ell} \!\!\Big).
\end{align*}
Since $C_d^{(d)}(\mathbf{P}^*)=1/\chi_{d}(\mathbf{p}_d^*)$, it remains to show that $A_\ell=0$ for every $1\leq \ell\leq d-1$. Using the definition of $C_{\ell}^{(d)}(\mathbf{P}^*)$ from~\eqref{eq:35}, it follows that
\begin{equation*}
A_{\ell} =  \bigg( 1-\sum_{m=\ell+1}^{d}C_m^{(d)}(\mathbf{P}^*)\cdot \chi_{\ell}(\mathbf{p}^*_m) \bigg) \frac{\chi_{\ell}(\mathbf{p}^*_{\ell})}{\chi_{\ell}(\mathbf{p}^*_{\ell})} + \sum_{n=\ell+1}^d C_n^{(d)}(\mathbf{P}^*) \big( \chi_{\ell}(\mathbf{p}_n^*) - \chi_{\ell+1}(\mathbf{p}_n^*) \big) =0,
\end{equation*}
where the final equality follows from Lemma~\ref{lem:2}.
	
The maximising type $\mathbf{P}^*$ is obtained by repeated use of the Lagrange-multipliers method. For brevity, let $t(\mathbf{p}_d,\mathbf{p}_{d-1},\ldots,\mathbf{p}_1):=\sum_{n=1}^d C_n^{(d)}(\mathbf{P})\cdot H(\mathbf{p}_n)$. Note that $t(\mathbf{p}_d,\mathbf{p}_{d-1},\ldots,\mathbf{p}_1)$ depends on $\mathbf{p}_1$ only through the first term. More specifically, by~\eqref{eq:35},
\begin{equation*}
C_1^{(d)}(\mathbf{P})\cdot H(\mathbf{p}_1) = \bigg( 1-\sum_{m=2}^{d}C_m^{(d)}(\mathbf{P})\cdot \chi_{1}(\mathbf{p}_m) \bigg) \frac{H(\mathbf{p}_1)}{\chi_{1}(\mathbf{p}_{1})}.
\end{equation*}
The term in parenthesis is independent of $\mathbf{p}_{1}$, furthermore, as mentioned already in Section~\ref{sec:11}, the quotient $H(\mathbf{p}_1)/\chi_1(\mathbf{p}_1)$ is maximised precisely by $\mathbf{p}_1^*$ with value $s_1$.

The next step is to observe that $t(\mathbf{p}_d,\ldots,\mathbf{p}_2,\mathbf{p}_1^*)$ depends on $\mathbf{p}_2$ only through the first two terms. More specifically writing out these two terms, by~\eqref{eq:35},
\begin{equation*}
\bigg( 1-\sum_{m=3}^{d}C_m^{(d)}(\mathbf{P})\cdot \chi_{1}(\mathbf{p}_m) \bigg)\cdot s_1 + \bigg( 1-\sum_{m=3}^{d}C_m^{(d)}(\mathbf{P})\cdot \chi_{2}(\mathbf{p}_m) \bigg) \frac{ H(\mathbf{p}_2) - \chi_{1}(\mathbf{p}_2)\cdot s_1 }{\chi_{2}(\mathbf{p}_2)}.
\end{equation*}
The two terms in parenthesis are independent of $\mathbf{p}_2$, moreover, another use of the Lagrange-multipliers method shows that the quotient depending on $\mathbf{p}_2$ is maximised by $\mathbf{p}_2^*$ with value $s_2-s_1$. In general, at the $n$-th step one applies the Lagrange-multipliers method to the term in $t(\mathbf{p}_d,\ldots,\mathbf{p}_n,\mathbf{p}_{n-1}^*,\ldots,\mathbf{p}_1^*)$ that depends on $\mathbf{p}_n$.
This concludes the proof.
\end{proof}


\subsection{Proof of Theorem~\ref{thm:main2}}\label{sec:41}

Without the coordinate ordering condition~\eqref{eq:21}, the study of Bara\'nski sponges is usually much more technical than the Gatzouras--Lalley case. However, for our box counting argument only one extra natural step is required.

The $\delta$-stopping of $\ii=(\iu_1\iu_2\ldots)\in\Sigma$ in the $n$-th coordinate (for $n=1,\ldots,d$) is the same as in~\eqref{eq:34} with the slightly modified notation:
\begin{equation*}
\prod_{\ell=1}^{L_{\delta}(\ii,n)} \lambda_n\big(\iu_{\ell,n}\big) \leq \delta < \prod_{\ell=1}^{L_{\delta}(\ii,n)-1}\lambda_n\big(\iu_{\ell,n}\big),
\end{equation*}
where $\iu_{\ell,n}$ denotes the $n$-th coordinate of $\iu_{\ell}$. The symbolic $\delta$-approximate cube containing $\ii\in\Sigma$ is the same as before:
\begin{equation*}
B_\delta(\ii) = \big\{ \jj\in\Sigma:\, \iu_{\ell,n}=\ju_{\ell,n}\; \text{ for every } \ell =1,\ldots,L_{\delta}(\ii,n)  \text{ and } n=1,\ldots,d\big\}.
\end{equation*}
Also, the approximate cubes partition $\Sigma$, and their images by the natural projection $\pi$ give an optimal $\delta$-cover of the attractor. Without the coordinate ordering condition, we do not know how the $L_{\delta}(\ii,n)$ compare to each other for a specific $B_{\delta}(\ii)$ like we did in~\eqref{eq:30} for the Gatzouras--Lalley case. Therefore, we sort the approximate cubes first.

Recall, $\mathrm{Sym}(\{1,\ldots, d\})$ denotes the symmetric group on the set of coordinates $\{1,\ldots,d\}$ and the notation for a permutation $\sigma$ from~\eqref{eq:24}. 
We say that a $\delta$-approximate cube $B_{\delta}(\ii)$ is \emph{$\sigma$-ordered} if 
\begin{equation*}
L_{\delta}(\ii,\sigma_d)\leq L_{\delta}(\ii,\sigma_{d-1})\leq \ldots\leq L_{\delta}(\ii,\sigma_1).
\end{equation*}
Potentially $B_{\delta}(\ii)$ can be $\sigma$-ordered for different permutations if the $\delta$-stopping is equal in multiple coordinates, but we will see in a moment that this is never a dominant box counting class. Let $\mathcal{B}_\delta(\sigma)$ denote the set of $\sigma$-ordered $\delta$-approximate cubes.

For a fixed $\sigma\in\mathrm{Sym}(\{1,\ldots, d\})$ at every scale $\delta$, the $\delta$-stoppings within $\mathcal{B}_\delta(\sigma)$ are ordered the same way, hence, we can identify $B_\delta(\ii)\in\mathcal{B}_\delta(\sigma)$ with the sequence
\begin{equation*}
\big\{ \iu_{1,\sigma_d},\ldots,\iu_{L_{\delta}(\ii,\sigma_d),\sigma_d}
\,;\,
\iu_{L_{\delta}(\ii,\sigma_d)+1,\sigma_{d-1}},\ldots,\iu_{L_{\delta}(\ii,\sigma_{d-1}),\sigma_{d-1}}
\,;\,\ldots\,;\, 
\iu_{L_{\delta}(\ii,\sigma_2)+1,\sigma_1},\ldots,\iu_{L_{\delta}(\ii,\sigma_1),\sigma_1}
\big\},
\end{equation*}
where a block is empty whenever $L_{\delta}(\ii,\sigma_n) = L_{\delta}(\ii,\sigma_{n+1})$. The type for an $\ii\in\mathcal{B}_\delta(\sigma)$ has the form 
$
\tau_{\delta}(\ii) = \big( \tau_{\delta}(\ii,\sigma_d)\,;\, \tau_{\delta}(\ii,\sigma_{d-1})\,;\,\ldots \,;\, \tau_{\delta}(\ii,\sigma_1) \big),
$
where $\tau_{\delta}(\ii,\sigma_n)$ is equal to 
\begin{equation*}
 \frac{1}{L_{\delta}(\ii,\sigma_n)-L_{\delta}(\ii,\sigma_{n+1})} \Big( \#\big\{L_{\delta}(\ii,\sigma_{n+1})+1\leq\ell\leq L_{\delta}(\ii,\sigma_n):\, \Pi(\iu_{\ell},\{\sigma_1^n\})=\underline{j}\big\}  \Big)_{\underline{j}\in\Pi(\mathcal{I};\{\sigma_1^n\})}
\end{equation*}
for $1\leq n\leq d$, where $\{\sigma_1^n\}=\{\sigma_1,\ldots,\sigma_n\}$. 
If $L_{\delta}(\ii,\sigma_n) = L_{\delta}(\ii,\sigma_{n+1})$, then the corresponding $C_n^{(d)}(\tau_{\delta}(\ii))=0$. Hence, from~\eqref{eq:33} of Proposition~\ref{prop:1} it follows that such a type can never be a dominant box counting type. Moreover, the number of different types with at least one empty block is certainly bounded from above by $o(\delta^{-1})$. Therefore, from the point of view of determining the box dimension, we can simply discard the approximate cubes in these type classes.

As a result, for any fixed $\sigma\in\mathrm{Sym}(\{1,\ldots, d\})$, we are essentially back in the GL case and can repeat the same argument. Within each $\mathcal{B}_\delta(\sigma)$ there is a dominant box counting type $\mathbf{P}^*_\delta(\sigma)=(\mathbf{p}_{\delta,d}^*;\mathbf{p}_{\delta,d-1}^*;\ldots; \mathbf{p}_{\delta,1}^*)$ which consists of probability vectors $\mathbf{p}_{\delta,n}^*$ on the index set $\Pi(\mathcal{I};\{\sigma_1^n\})$. As $\delta\to 0$, these vectors $\mathbf{p}_{\delta,n}^*$ converge to the ones defined by the equations in Theorem~\ref{thm:main2}. This is the limiting dominant type $\mathbf{P}^*(\sigma)$ which satisfies $\sum_{n=1}^d C_n^{(d)}(\mathbf{P}^*(\sigma))\cdot H(\mathbf{p}^*_n(\sigma))=s_d(\sigma)$. Thus, $\#T_\delta(\mathbf{P}^*_\delta(\sigma))\approx\delta^{-s_d(\sigma)+o(1)}$. Since there are just $d!$ different $\sigma$-orderings, we conclude that $\dim_{\mathrm B} \Lambda_{d} = \max_{\sigma\in\mathrm{Sym}(\{1,\ldots, d\})} s_d(\sigma)$.

\section{Further discussion}\label{sec:5}

This section provides some additional context to the results.

First consider Bedford--McMullen (or Sierpi\'nski) sponges. They are special cases of GL sponges because the diagonal matrices $A_{\iu}$ defining the maps $S_{\iu}$ are all the same and independent of $\iu$. Let $1>\lambda_1>\lambda_2>\ldots>\lambda_d>0$ denote the diagonal entries. Similarly to the homogeneous self-similar case, recall Remark~\ref{remark1}, the $\delta$-stoppings are independent of $\ii$ and $L_{\delta}(\ii,n)\approx \log \delta / \log \lambda_n$ for $1\leq n\leq d$. Hence, $C_n^{(d)}(\mathbf{P}) = 1/\log \lambda_n - 1/\log \lambda_{n+1}$ regardless of $\mathbf{P}$. Thus, Proposition~\ref{prop:1} implies that all we need to maximise in~\eqref{eq:33} is $H(\mathbf{p}_n)$ which is equal to $\log \#\mathcal{I}_n$ (attained by the uniform vector on the set $\mathcal{I}_n$). This is the formula obtained by Kenyon and Peres~\cite{KenyonPeres_ETDS96}.  

Another setup to which the method can be applied to is if we consider GL carpets in two dimensions defined by lower triangular matrices instead of diagonal matrices~\cite{BaranskiTriag_2008,KolSimon_TriagGL2019}. In this case the image of $[0,1]^2$ under any map of the IFS is a parallelogram with two vertical sides parallel with the $y$-axis. A simple lemma~\cite[Lemma 1.3]{KolSimon_TriagGL2019} states that the slope of the iterates of these parallelograms remain uniformly bounded. Hence, there is a uniform constant $C$ (depending only on the IFS) such that the image by $\pi$ of any $\delta$-approximate square on $\Lambda$ can be covered by at most $C$ squares of diameter $\delta$. As a result, $\#\mathcal{B}_\delta\approx N_{\delta}(\Lambda)$ still holds, so the box dimension remains unchanged.

Our variational formula~\eqref{eq:23} also provides a very clear argument for one of the necessary and sufficient conditions for the Hausdorff and box dimensions of GL carpets to agree. Gatzouras and Lalley~\cite{GatzourasLalley92} proved that the Hausdorff dimension satisfies the variational formula
\begin{equation}\label{eq:50}
\dim_{\mathrm H}\Lambda_2 = \max_{\mathbf{p}\in\mathcal{P}_{\mathcal{I}_2}}\; \frac{H(\mathbf{p})}{\chi_2(\mathbf{p})}  +
\left( 1- \frac{\chi_1(\mathbf{p})}{\chi_2(\mathbf{p})} \right) \frac{H(\mathbf{q}_{\mathbf{p}})}{\chi_1(\mathbf{q}_{\mathbf{p}})},
\end{equation}
where $\mathbf{q}_{\mathbf{p}}=(q_1,\ldots,q_{\#\mathcal{I}_1})$ denotes the probability vector on $\mathcal{I}_1$ defined by $q_i=\sum_{j\in\mathcal{I}(i)} p_{(i,j)}$. Comparing this with~\eqref{eq:23}, we immediately see that
\begin{equation*}
\dim_{\mathrm H} \Lambda_2 = \dim_{\mathrm B} \Lambda_2 \;\Longleftrightarrow\; \mathbf{q}_{\mathbf{p}_2^*} = \mathbf{p}_1^* \;\Longleftrightarrow\; \sum_{j\in\mathcal{I}(i)} \big(\lambda(i,j)\big)^{s_2-s_1} =1  \;\text{ for every } i\in\mathcal{I}_1.
\end{equation*}
This is referred to as the \emph{uniform fibre} case in the literature. The main result of Das and Simmons~\cite{das2017hausdorff} is that the variational formula~\eqref{eq:50} does not necessarily hold in higher dimensions. Instead, one needs to consider a wider class of measures, called \emph{pseudo-Bernoulli measures}, which are not invariant. 

The expression being maximised in~\eqref{eq:50} is a special case of the \emph{Ledrappier--Young formula} which holds in much higher generality for measures on self-affine sets~\cite{BARANYAnti201788,BaranyRams_TransAMS18,feng2020LYformulaArxiv,FengHu09} and has been a key technical tool in recent advancements in the dimension theory of self-affine sets and measures, see \cite{BaranyHochmanRapaport,hochmanRapaport2019arxiv,MorrisShmerkin_TransAMS19,Rapaport_TransAMS18} to name a few. In light of our result, it is natural to ask the following.
\begin{question}
Does a Ledrappier--Young like formula~\eqref{eq:23} hold more generally for the box dimension of self-affine sets on the plane? What about higher dimensions?
\end{question}
For three dimensions, the formula would be to maximise the expression in~\eqref{eq:40}. The general argument itself is very flexible. If the optimal $\delta$-cover of a set has a clear symbolic representation, then by defining a proper space of types it seems plausible to apply the method. The Bara\'nski case shows that some ``orientation'' of the boxes also plays a role. 

Overlapping systems could be particularly interesting to study from this vantage point. This is because it is still an open problem whether the box dimension of self-affine sets always exists, regardless of overlaps. It does not exist for all \emph{sub}-self-affine sets introduced in~\cite{KaenmakiVilppolainen_2010}, see the very recent example of Jurga~\cite{Jurga_DimBsubselfaffine_arxiv}. Moreover, for self-similar sets there is the folklore conjecture that the only reason why its (box) dimension can drop below its similarity dimension~\eqref{eq:07} is if the system has exact overlaps.






\subsection*{Acknowledgment}

The author was supported by a \emph{Leverhulme Trust Research Project Grant} (RPG-2019-034). 

\bibliographystyle{abbrv}
\bibliography{biblio_methodtypes}

\begin{thebibliography}{10}

\bibitem{BARANSKIcarpet_2007}
K.~Bara{\'n}ski.
\newblock Hausdorff dimension of the limit sets of some planar geometric
  constructions.
\newblock {\em Advances in Mathematics}, 210(1):215 -- 245, 2007.

\bibitem{BaranskiTriag_2008}
K.~Bara{\'n}ski.
\newblock Hausdorff dimension of self-affine limit sets with an invariant
  direction.
\newblock {\em Discrete Continuous Dynamical Systems - A}, 21(4):1015--1023,
  2008.

\bibitem{BaranyHochmanRapaport}
B.~{B{\'a}r{\'a}ny}, M.~{Hochman}, and A.~{Rapaport}.
\newblock {Hausdorff dimension of planar self-affine sets and measures}.
\newblock {\em Inventiones mathematicae}, 216:601--659, 2019.

\bibitem{BARANYAnti201788}
B.~B\'ar\'any and A.~K\"{a}enm\"{a}ki.
\newblock Ledrappier--{Y}oung formula and exact dimensionality of self-affine
  measures.
\newblock {\em Advances in Mathematics}, 318:88 -- 129, 2017.

\bibitem{BaranyRams_TransAMS18}
B.~B{\'a}r{\'a}ny and M.~Rams.
\newblock Dimension maximizing measures for self-affine systems.
\newblock {\em Transactions of the American Mathematical Society},
  370:553--576, 2018.

\bibitem{Bedford84_phd}
T.~Bedford.
\newblock {\em Crinkly curves, Markov partitions and box dimensions in
  self-similar sets}.
\newblock PhD thesis, University of Warwick, 1984.

\bibitem{Bremaud2017MethodofTypes}
P.~Br{\'e}maud.
\newblock {\em The Method of Types}, pages 341--355.
\newblock Springer International Publishing, Cham, 2017.

\bibitem{Csiszar_98MethodofTypes}
I.~{Csisz{\'a}r}.
\newblock The method of types [information theory].
\newblock {\em IEEE Transactions on Information Theory}, 44(6):2505--2523,
  1998.

\bibitem{das2017hausdorff}
T.~Das and D.~Simmons.
\newblock The {H}ausdorff and dynamical dimensions of self-affine sponges: a
  dimension gap result.
\newblock {\em Inventiones mathematicae}, 210(1):85--134, 2017.

\bibitem{DemboZeitouniLDP}
A.~Dembo and O.~Zeitouni.
\newblock {\em Large {D}eviations {T}echniques and {A}pplications}, volume~38
  of {\em Stochastic Modelling and Applied Probability}.
\newblock Springer-Verlag Berlin Heidelberg, 2010.

\bibitem{falconer_1988}
K.~J. Falconer.
\newblock The {H}ausdorff dimension of self-affine fractals.
\newblock {\em Mathematical Proceedings of the Cambridge Philosophical
  Society}, 103(2):339--350, 1988.

\bibitem{falconer_1992}
K.~J. Falconer.
\newblock The dimension of self-affine fractals {II}.
\newblock {\em Mathematical Proceedings of the Cambridge Philosophical
  Society}, 111(1):169–--179, 1992.

\bibitem{FalconerBook}
K.~J. Falconer.
\newblock {\em Fractal Geometry: Mathematical Foundations and Applications}.
\newblock 3rd Ed., John Wiley {\&} Sons, Hoboken, NJ, 2014.

\bibitem{feng2020LYformulaArxiv}
D.-J. {Feng}.
\newblock {Dimension of invariant measures for affine iterated function
  systems}.
\newblock {\em arXiv e-prints},  arXiv:1901.01691, 2019.

\bibitem{FengHu09}
D.-J. Feng and H.~Hu.
\newblock Dimension theory of iterated function systems.
\newblock {\em Communications on Pure and Applied Mathematics},
  62(11):1435--1500, 2009.

\bibitem{FengWang2005}
D.-J. Feng and Y.~Wang.
\newblock A class of self-affine sets and self-affine measures.
\newblock {\em Journal of Fourier Analysis and Applications}, 11(1):107--124,
  2005.

\bibitem{Fraser12Boxlike}
J.~M. Fraser.
\newblock On the packing dimension of box-like self-affine sets in the plane.
\newblock {\em Nonlinearity}, 25(7):2075--2092, 2012.

\bibitem{FraserHowroyd_AnnAcadSciFennMath17}
J.~M. Fraser and D.~Howroyd.
\newblock Assouad type dimensions for self-affine sponges.
\newblock {\em Annales Academi\ae{} Scientiarum Fennic\ae{}}, 42:149--174,
  2017.

\bibitem{fraserJurga2019SpongeboxArXiv}
J.~M. {Fraser} and N.~{Jurga}.
\newblock {The box dimensions of exceptional self-affine sets in
  $\mathbb{R}^3$}.
\newblock {\em arXiv e-prints}, arXiv:1907.07593, 2019.

\bibitem{fraser_shmerkin_2016}
J.~M. Fraser and P.~Shmerkin.
\newblock On the dimensions of a family of overlapping self-affine carpets.
\newblock {\em Ergodic Theory and Dynamical Systems}, 36(8):2463--2481, 2016.

\bibitem{GatzourasLalley92}
D.~Gatzouras and S.~P. Lalley.
\newblock Hausdorff and box dimensions of certain self-affine fractals.
\newblock {\em Indiana University Mathematics Journal}, 41(2):533--568, 1992.

\bibitem{hochmanRapaport2019arxiv}
M.~{Hochman} and A.~{Rapaport}.
\newblock {Hausdorff Dimension of Planar Self-Affine Sets and Measures with
  Overlaps}.
\newblock {\em arXiv e-prints}, arXiv:1904.09812, 2019.

\bibitem{Jurga_DimBsubselfaffine_arxiv}
N.~{Jurga}.
\newblock {Non-existence of the box dimension for dynamically invariant sets}.
\newblock {\em arXiv e-prints}, arXiv:2102.04375, 2021.

\bibitem{KaenmakiVilppolainen_2010}
A.~K\"{a}enm\"{a}ki and M.~Vilppolainen.
\newblock Dimension and measures on sub-self-affine sets.
\newblock {\em Monatshefte für Mathematik}, 161:271 -- 293, 2010.

\bibitem{KenyonPeres_ETDS96}
R.~Kenyon and Y.~Peres.
\newblock Measures of full dimension on affine-invariant sets.
\newblock {\em Ergodic Theory and Dynamical Systems}, 16(2):307--–323, 1996.

\bibitem{KolSimon_TriagGL2019}
I.~Kolossv{\'{a}}ry and K.~Simon.
\newblock Triangular {G}atzouras{\textendash}{L}alley-type planar carpets with
  overlaps.
\newblock {\em Nonlinearity}, 32(9):3294--3341, 2019.

\bibitem{mcmullen84}
C.~McMullen.
\newblock The {H}ausdorff dimension of general {S}ierpi{\'n}ski carpets.
\newblock {\em Nagoya Mathematical Journal}, 96:1--9, 1984.

\bibitem{MorrisShmerkin_TransAMS19}
I.~D. Morris and P.~Shmerkin.
\newblock On equality of {H}ausdorff and affinity dimensions, via self-affine
  measures on positive subsystems.
\newblock {\em Transactions of the American Mathematical Society},
  371:1547--1582, 2019.

\bibitem{pardo-simon}
L.~Pardo-Sim\'on.
\newblock Dimensions of an overlapping generalization of {B}ara\'nski carpets.
\newblock {\em Ergodic Theory and Dynamical Systems}, pages 1--31, 2017.

\bibitem{Rapaport_TransAMS18}
A.~Rapaport.
\newblock On self-affine measures with equal {H}ausdorff and {L}yapunov
  dimensions.
\newblock {\em Transactions of the American Mathematical Society},
  370:4759--4783, 2018.

\end{thebibliography}

\end{document}